\documentclass[12pt,reqno]{article}
\usepackage{amsfonts,amsmath,amsthm,amssymb,latexsym}
\date{}
\newcommand{\R}{{\mathbb R}}
\newcommand{\const}{\mathrm{const}}

\newtheorem{lemma}{Lemma}
\theoremstyle{definition}
\newtheorem{remark}{Remark}

\title{On the structure of weak solutions to the Riemann problem for degenerate nonlinear diffusion equation}

\author{Evgeny~Yu.~Panov
\\ \small
Yaroslav-the-Wise Novgorod State University, Veliky Novgorod, Russian Federation, \\ \small Research and Development Center, Veliky Novgorod, Russian Federation and \\ \small Peoples' Friendship University of Russia (RUDN University), Moscow, Russian Federation
}

\begin{document}
\maketitle

\begin{abstract}
We find an explicit form of weak solutions to a Riemann problem for a degenerate semilinear parabolic equation with piecewise constant diffusion coefficient. It is demonstrated that the phase transition lines (free boundaries) correspond to the minimum point of some strictly convex function of a finite number of variables. In the limit as number of phases tend to infinity we obtain a variational formulation of self-similar solution with an arbitrary nonnegative diffusion function.
\end{abstract}

\bigskip
In a half-plane $t>0$, $x\in\R$, we consider a semilinear parabolic equation
\begin{equation}\label{1}
u_t=(a^2(u)u_x)_x=A(u)_{xx},
\end{equation}
where $A'(u)=a^2(u)$, $a(u)\in L^\infty(\R)$, $a(u)\ge 0$. The change $U=A(u)$ reduces (\ref{1}) to the equation
$$
b(U)_t=U_{xx},
$$
with $b(U)=A^{-1}(U)$ being a strictly increasing and possibly discontinuous function.

The same reduction was used in \cite[Chapter~5,\S~9]{LSU} for a Stefan problem. By the similar methods as in \cite{LSU} one can establish the existence and uniqueness of a bounded weak solution (understood in the sense of distributions) of the Cauchy problem for equation (\ref{1}). In the case of piecewise smooth weak solution the following conditions should be fulfilled on discontinuity lines $x=x(t)$:
\begin{equation}\label{RG}
[A(u)]=0, \quad [u]x'(t)+[A(u)_x]=0,
\end{equation}
where $[v]=v(t,x(t)+)-v(t,x(t)-)$ denotes a jump of a function $v=v(t,x)$ on the discontinuity line. Conditions (\ref{RG}) are derived by the application of the distribution $0=u_t-A(u)_{xx}$ to an arbitrary test function and integration by parts with the help of Green formula. As is easy to verify, these conditions together with the requirement that $u(t,x)$ is a classic solution of (\ref{1}) in the smoothness domains are equivalent to the statement that $u=u(t,x)$ is a weak solution of equation (\ref{1}).

We will study the Cauchy problem for equation (\ref{1}) with a Riemann initial data
\begin{equation}\label{2}
u(0,x)=\left\{\begin{array}{lr} u_-, & x<0, \\ u_+, & x>0. \end{array}\right.
\end{equation}
As was mentioned above, a weak solution of (\ref{1}), (\ref{2}) is unique. By the uniqueness and the invariance of the problem under the transformation group $(t,x)\to (\lambda^2 t, \lambda x)$, $\lambda\in\R$, $\lambda\not=0$, a weak solution of problem (\ref{1}), (\ref{2}) is self-similar: $u(t,x)=v(\xi)$, $\xi=x/\sqrt{t}$. In the case of the heat equation $u_t=a^2 u_{xx}$ a self-similar solution $u=v(\xi)$ must satisfy the linear ODE $a^2v''=-\xi v'/2$, the general solution of which is
$v=C_1F(\xi/a)+C_2$, $C_1,C_2=\const$, where $$F(\xi)=\frac{1}{2\sqrt{\pi}}\int_{-\infty}^\xi e^{-s^2/4}ds$$ is the error function. In particular, the solution of Riemann problem (\ref{1}), (\ref{2}) can be found if we choose $C_2=u_-$, $C_1=u_+-u_-$. We consider now the nonlinear case assuming that the diffusion coefficient is a piecewise constant function, $a(u)=a_k$ for $u_k<u<u_{k+1}$, $k=0,\ldots,n$, where
$$u_-=u_0<u_1<\cdots<u_n<u_{n+1}=u_+$$ (since equation (\ref{1}) is invariant under the change $x\to-x$, we can suppose that $u_+>u_-$), and $a_{k+1}\not=a_k$, $k=0,\ldots,n-1$. Let us begin with the nondegenerate case $a_k>0$, $k=0,\ldots,n$. We are going to find a weak solution $u=v(\xi)$ of (\ref{1}), (\ref{2}) as a piecewise smooth function
\begin{align}\label{3}
v(\xi)=u_k+\frac{u_{k+1}-u_k}{F(\xi_{k+1}/a_k)-F(\xi_k/a_k)}(F(\xi/a_k)-F(\xi_k/a_k)), \\ \nonumber
\xi_k<\xi<\xi_{k+1}, \ k=0,\ldots,n,
\end{align}
where $$-\infty=\xi_0<\xi_1<\cdots<\xi_n<\xi_{n+1}=+\infty$$ and we agree that $F(-\infty)=0$, $F(+\infty)=1$.
The parabolas $\xi=\xi_k$, $k=1,\ldots,n$, should be determine by conditions (\ref{RG}). Observing that the piecewise linear function $A(u)$ is strictly increasing, we conclude that the condition $[A(u)]=0$ reduces to the continuity condition $[u]=0$. We see that $u=u_k$ on the phase transition lines $\xi=\xi_k$, and these lines are weak discontinuities of $u(t,x)$. The second condition in (\ref{RG}) means that $[A(u)'_\xi]=0$ (i.e., $A(u)\in C^1$). In view of (\ref{3}) these conditions reduce to the equalities
\begin{equation}\label{4}
\frac{a_k(u_{k+1}-u_k)F'(\xi_k/a_k)}{F(\xi_{k+1}/a_k)-F(\xi_k/a_k)}=
\frac{a_{k-1}(u_k-u_{k-1})F'(\xi_k/a_{k-1})}{F(\xi_k/a_{k-1})-F(\xi_{k-1}/a_{k-1})}, \quad k=1,\ldots,n.
\end{equation}
To prove that the nonlinear system (\ref{4}) admits a solution, we observe that it is equivalent to the equality
$\nabla E(\bar\xi)=0$, where the function
\begin{equation}\label{5}
E(\bar\xi)=-\sum_{k=0}^n  (a_k)^2(u_{k+1}-u_k)\ln (F(\xi_{k+1}/a_k)-F(\xi_k/a_k)), \quad \bar\xi=(\xi_1,\ldots,\xi_n)\in\Omega,
\end{equation}
defined in the convex open domain $\Omega$ given by the inequalities $\xi_1<\cdots<\xi_n$. Obviously, $E(\bar\xi)\in C^\infty(\Omega)$. Notice that for all $k=0,\ldots,n$
\begin{equation}\label{6}
\ln (F(\xi_{k+1}/a_k)-F(\xi_k/a_k))<0.
\end{equation}
In particular, $E(\bar\xi)>0$. We will call the function $E(\bar\xi)>0$ \textbf{the entropy}, it depends only on discontinuities of the function $v(\xi)$.

\begin{lemma}\label{lem1}
The sets $E(\bar\xi)\le c$ are compact for each constant $c\in\R$.
\end{lemma}

\begin{proof}
If $E(\bar\xi)\le c$ then it follows from (\ref{6}) that for all $k=0,\ldots,n$
\begin{align*}
-(a_k)^2(u_{k+1}-u_k)\ln (F(\xi_{k+1}/a_k)-F(\xi_k/a_k))\le E(\bar\xi)\le c,
\end{align*}
which implies the estimate
\begin{equation}\label{7}
F(\xi_{k+1}/a_k)-F(\xi_k/a_k)\ge \delta\doteq\exp(-c/m)>0,
\end{equation}
where $\displaystyle m=\min_{k=0,\ldots,n}(a_k)^2(u_{k+1}-u_k)>0$. In the particular cases $k=0,n$ these inequalities imply that $F(\xi_1/a_0)\ge\delta$, $F(-\xi_n/a_n)=1-F(\xi_n/a_n)\ge\delta$, therefore
$-r\le\xi_1<\xi_n\le r$, where $r$ is a positive constant satisfying the condition $\max(F(-r/a_0),F(-r/a_n))\le\delta$. Since the remaining coordinates of $\bar\xi$ are situated between $\xi_1$ and $\xi_n$, we find that ${\|\bar\xi\|_\infty\le r}$. Further, since $F'(x)=\frac{1}{2\sqrt{\pi}}e^{-x^2/4}<1$, the function $F(x)$ is Lipschitz with constant $1$ and it follows from (\ref{7}) that
$$
(\xi_{k+1}-\xi_k)/a_k\ge F(\xi_{k+1}/a_k)-F(\xi_k/a_k)\ge \delta, \quad k=1,\ldots,n-1.
$$
We find that $\xi_{k+1}-\xi_k\ge\delta_1=\delta\min a_k$. We conclude tat the set $E(\bar\xi)\le c$ lies in the compact set
$$
K=\{ \ \bar\xi=(\xi_1,\ldots,\xi_n)\in\R^n \ | \ \|\bar\xi\|_\infty\le r, \ \xi_{k+1}-\xi_k\ge\delta_1 \ \forall k=1 ,\ldots,n-1 \ \}.
$$
By the continuity of $E(\bar\xi)$ the set $E(\bar\xi)\le c$ is a closed subset of $K$ and therefore is compact.
\end{proof}
We take $c>N\doteq\inf E(\bar\xi)$. Then the set $E(\bar\xi)\le c$ is not empty. By Lemma~\ref{lem1} this set is compact and therefore the continuous function $E(\bar\xi)$ reaches the minimal value on it, which is evidently equal to $N$. We proved the existence of global minimum $E(\bar\xi_0)=\min E(\bar\xi)$. At the minimum point $\bar\xi_0$ the required condition $\nabla E(\bar\xi_0)=0$ holds and we conclude that system (\ref{4}) has a solution.

The uniqueness of this solution follows from the uniqueness of a weak solution of our Riemann problem. Alternatively, this uniqueness can be derived from strict convexity of the entropy function $E(\bar\xi)$ (so that this function can have at most one extremum in $\Omega$). In view of representation (\ref{5}) the strict convexity of the entropy easily follows from the lemma below.

\begin{lemma}\label{lem2}
The function $P(x,y)=-\ln (F(x)-F(y))$ is strictly convex in the half-plane $x>y$.
\end{lemma}

\begin{proof}
The function $P(x,y)$ is infinitely differentiable in the domain $x>y$. To prove the lemma, we need to establish that the Hessian $D^2 P$ is positive definite at every point. By the direct computation we find
\begin{align*}
\frac{\partial^2}{\partial x^2} P(x,y)=\frac{(F'(x))^2-F''(x)(F(x)-F(y))}{(F(x)-F(y))^2}, \\
\frac{\partial^2}{\partial y^2} P(x,y)=\frac{(F'(y))^2-F''(y)(F(y)-F(x))}{(F(x)-F(y))^2}, \
\frac{\partial^2}{\partial x\partial y} P(x,y)=-\frac{F'(x)F'(y)}{(F(x)-F(y))^2}.
\end{align*}
We have to prove positive definiteness of the matrix $Q=(F(x)-F(y))^2 D^2 P(x,y)$ with the components
\begin{align*}
Q_{11}=(F'(x))^2-F''(x)(F(x)-F(y)), \\ Q_{22}=(F'(y))^2-F''(y)(F(y)-F(x)), \ Q_{12}=Q_{21}=-F'(x)F'(y).
\end{align*}
Since $F'(x)=e^{-x^2/4}$, then $F''(x)=-\frac{x}{2}F'(x)$ and the diagonal elements of this matrix can be written in the form
\begin{align*}
Q_{11}=F'(x)(\frac{x}{2}(F(x)-F(y))+F'(x))= \\ F'(x)(\frac{x}{2}(F(x)-F(y))+(F'(x)-F'(y)))+F'(x)F'(y), \\
Q_{22}=F'(y)(\frac{y}{2}(F(y)-F(x))+(F'(y)-F'(x)))+F'(x)F'(y).
\end{align*}
By Cauchy mean value theorem there exists such a value $z\in (y,x)$ that
$$
\frac{F'(x)-F'(y)}{F(x)-F(y)}=\frac{F''(z)}{F'(z)}=-z/2.
$$
Therefore,
\begin{align*}
Q_{11}=F'(x)(F(x)-F(y))(x-z)/2+F'(x)F'(y), \\ Q_{22}=F'(y)(F(x)-F(y))(z-y)/2+F'(x)F'(y),
\end{align*}
and it follows that $Q=R_1+F'(x)F'(y)R_2$, where $R_1$ is a diagonal matrix with the positive diagonal elements
$F'(x)(F(x)-F(y))(x-z)/2$, $F'(y)(F(x)-F(y))(z-y)/2$ while $R_2=\left(\begin{smallmatrix} 1 & -1 \\ -1 & 1\end{smallmatrix}\right)$. Since $R_1>0$, $R_2\ge 0$, then the matrix $Q>0$, as was to be proved.
\end{proof}

Now we consider the case when some of the coefficients $a_k=0$. If $k=0$ or $k=n$, the structure (\ref{3}) of the solution remains the same but for $\xi<\xi_1$ (respectively, for $\xi>\xi_n$) the solution becomes constant: $v\equiv u_-$ ($v\equiv u_+$). Moreover, the discontinuity $\xi=\xi_1$ ($\xi=\xi_n$) is now strong, and condition (\ref{RG}) implies the following relation of Stefan type
$$
-(u_1-u_-)\xi_1/2=\frac{a_1(u_2-u_1)F'(\xi_1/a_1)}{F(\xi_2/a_1)-F(\xi_1/a_1)},
$$
respectively,
$$
(u_+-u_n)\xi_n/2=\frac{a_{n-1}(u_n-u_{n-1})F'(\xi_n/a_{n-1})}{F(\xi_n/a_{n-1})-F(\xi_{n-1}/a_{n-1})}.
$$
These relations replace, respectively, the first and the last equality in (\ref{4}).
Notice that the first condition $[A(u)]=0$ in (\ref{RG}) is satisfied because the function $A(u)$ is constant on the segment $[u_-,u_1]$ (on $[u_n,u_+]$).
Replacing the first term in sum (\ref{5}) by $(u_1-u_-)(\xi_1)^2/4$ (the last term by $(u_+-u_n)(\xi_n)^2/4$), we obtain that all the required conditions on the lines $\xi=\xi_k$, $k=1,\ldots,n$, reduce again to the condition $\nabla E=0$, and the solution is uniquely determined by the minimum point of the entropy $E(\bar\xi)$. Remark that after the described reduction the entropy function remains to satisfy the statements of Lemma~\ref{lem1}, Lemma~\ref{lem2}.

The case when the diffusion degenerates in an inner interval, i.e., when $a_k=0$ for
$0<k<n$, is more complicated. In this case two weak discontinuities $\xi=\xi_k,\xi_{k+1}$ merge into one strong discontinuity $\xi=\xi_k$ with the limit values $v(\xi_k-)=u_k$, $v(\xi_k+)=u_{k+1}$, (one just need to put $\xi_{k+1}=\xi_k$ in formula (\ref{3})).
Condition (\ref{RG}) on the line $\xi=\xi_k$ reduces to the equality
$$
-(u_{k+1}-u_k)\xi_k/2=\frac{a_{k+1}(u_{k+2}-u_{k+1})F'(\xi_k/a_{k+1})}{F(\xi_{k+2}/a_{k+1})-F(\xi_k/a_{k+1})}-
\frac{a_{k-1}(u_k-u_{k-1})F'(\xi_k/a_{k-1})}{F(\xi_k/a_{k-1})-F(\xi_{k-1}/a_{k-1})}.
$$
This condition replaces two conditions with numbers $k,k+1$ in system (\ref{4}), where we also agree that $\xi_{k+1}=\xi_k$. Generally, we get the system of $n-l$ equations with $n-l$ unknowns, where $l$ is a number of inner intervals $(u_k,u_{k+1})$ with degenerate diffusion ($a_k=0$). As above, we find that solutions of this system coincide with critical points of the entropy function  $E(\bar\xi)$ of $n-l$ variables defined by the expression
\begin{align}\label{ent}
E(\xi_1,\ldots,\xi_n)=-\sum_{k=0,\ldots,n, a_k>0} (a_k)^2(u_{k+1}-u_k)\ln (F(\xi_{k+1}/a_k)-F(\xi_k/a_k)) \nonumber\\ +\sum_{k=0,\ldots,n, a_k=0} (u_{k+1}-u_k)(\xi_k)^2/4,
\end{align}
where variables $\xi_k$, $\xi_{k+1}$ are identified whenever $a_k=0$, $0<k<n$.

In the same way as in the nondegenerate case we can easily prove that this function is strictly convex and reaches the unique minimal value $E(\bar\xi_0)$. Coordinates of the minimum point $\bar\xi_0$ determine the unique weak solution of problem  (\ref{1}), (\ref{2}).

\begin{remark}
Adding to the entropy (\ref{ent}) the constant $$\sum_{k=0,\ldots,n, a_k>0} (a_k)^2(u_{k+1}-u_k)\ln ((u_{k+1}-u_k)/a_k),$$ we obtain the alternative variant of the entropy
\begin{align}\label{ent1}
E_1(\bar\xi)=-\sum_{k=0,\ldots,n, a_k>0} (a_k)^2(u_{k+1}-u_k)\ln\left(\frac{F(\xi_{k+1}/a_k)-F(\xi_k/a_k)}{(u_{k+1}-u_k)/a_k}\right) \nonumber\\ +\sum_{k=0,\ldots,n, a_k=0} (u_{k+1}-u_k)(\xi_k)^2/4.
\end{align}
If we consider the values $a_k$ as a piecewise constant approximation of an arbitrary diffusion function $a(u)\ge 0$ then, passing in (\ref{ent1}) to the limit as $\max(u_{k+1}-u_k)\to 0$, we find that the entropy $E_1(\bar\xi)$ turns into the variational functional
$$
J(\xi)=-\int_{\{u\in [u_-,u_+], a(u)>0\}} (a(u))^2\ln(F'(\xi(u)/a(u))\xi'(u))du+\int_{\{u\in [u_-,u_+], a(u)=0\}} (\xi(u))^2/4du,
$$
where $\xi(u)$ is an increasing function on $[u_-,u_+]$, which is expected to be the inverse function to a self-similar solution $u=u(\xi)$ to problem (\ref{1}), (\ref{2}). Taking into account that
$$
\ln(F'(\xi(u)/a(u))\xi'(u))=\ln F'(\xi(u)/a(u))+\ln\xi'(u)=-\frac{(\xi(u))^2}{4a^2(u)}+\ln\xi'(u),
$$
we may simplify the expression for the functional $J(\xi)$
\begin{equation}\label{var}
J(\xi)=-\int_{u_-}^{u_+} (a(u))^2\ln(\xi'(u))du+\frac{1}{4}\int_{u_-}^{u_+} (\xi(u))^2du.
\end{equation}
We see that this functional is strictly convex. As is easy to verify, the corresponding Euler-Lagrange equation has the form
\begin{equation}\label{EL}
\xi(u)/2+((a(u))^2/\xi'(u))'=0.
\end{equation}
Since $u'(\xi)=1/\xi'(u)$, $u=u(\xi)$, we can transform (\ref{EL}) as follows $$\xi(u)/2+((a(u))^2 u'(\xi))'_u=0.$$ Multiplying this equation by $u'(\xi)$, we obtain the equation
$$
(a^2u')'=-\xi u'/2, \quad u=u(\xi),
$$
which is exactly our equation (\ref{1}) written in the self-similar variable.
We conclude that minimization of the functional (\ref{var}) gives a variational formulation of self-similar solutions of equation (\ref{1}).
\end{remark}

\section*{Acknowledgments}
The author thanks professor E.\,V.~Radkevich for fruitful discussions.
The research was supported by the Russian Science Foundation, grant 22-21-00344.

\end{document}